\documentclass{amsart}
\usepackage{color}
\usepackage[utf8]{inputenc}
\usepackage{amscd}
\usepackage{amsmath}
\usepackage[english]{babel}
\usepackage{amsfonts}
\usepackage{graphicx}
\usepackage{vmargin}
\usepackage{theoremref}
\usepackage{amssymb}
\usepackage{commath}
\usepackage{cancel}
\usepackage{lineno}
\usepackage{comment}
\usepackage{xcolor}
\usepackage{lineno}

\newtheorem{theorem}{Theorem}[section]
\newtheorem{lemma}[theorem]{Lemma}

\newtheorem{definition}[theorem]{Definition}
\newtheorem{proposition}[theorem]{Proposition}
\newtheorem{corollary}[theorem]{Corollary}

\newtheorem{remark}[theorem]{Remark}
 \usepackage{color} 
\newcommand{\R}{\mathbb{R}}
\newcommand{\B}{\mathcal{B}}

\newcommand{\N}{\mathbb{N}}
\newcommand{\x}{\overline{\xi}}
\newcommand{\n}{\overline{\eta}}
\newcommand{\dint}{\displaystyle\int}

\numberwithin{equation}{section}

\title[]{A nonautonomous $C^r-$topological equivalence involving contractions and unbounded nonlinearities.}
\author[]{\'Alvaro Casta\~neda}
\address{Universidad de Chile, Departamento de Matem\'aticas. Casilla 653, Santiago, Chile}

\author[]{Fernanda Torres}
\address{Pontificia Universidad Cat\'olica de Chile, Departamento de Matem\'aticas Santiago, Chile}
\email{castaneda@uchile, fztorres@uc.cl}
\thanks{This research has been supported by FONDECYT Regular 1200653}
\date{\today}

\begin{document}

\maketitle

\begin{abstract}
We study the smoothness of the topological equivalence between a linear equation and its nonlinear perturbation, which is regarded as unbounded. To the best of our knowledge, it has not previously been considered such study in the literature. Therefore, the main result of this work copes with this lack, that is, it is shown, on the positive half line, that such topological equivalence is of class $C^r (r \geq 1)$ when the linear part is a uniform contraction and the nonlinearities considered are unbounded.
\end{abstract}

\section{introduction}

A fundamental tool for the study of the local behavior of nonlinear dynamical systems is Hartman–Grobman’s Theorem \cite[Theorem I]{Hartman1}. This result establishes the existence of a local topological conjugacy between the solutions of a nonlinear system with its linearization around an hyperbolic equilibrium, that is, the dynamics are topologically the same in a neighborhood of the equilibrium point. The analysis of the global behavior of nonlinear systems begins when C. Pugh \cite{Pugh} studied a particular case of the Linearization Theorem focused on linear systems with bounded and Lipschitz perturbations, allowing the construction of a global homeomorphism.

\subsection{Nonautonomous Linearization} Since K.J. Palmer in \cite{Palmer}, inspired by the global approach carried out by C. Pugh, introduced the nonautonomous linearization theorem for differentiable systems, the goal of showing that such linearization has some class of differentiability has been of wide interest. This goal has gained prominence in the last decade as a result of a large number of works that have dealt with this topic from a variety of perspectives. To the best of our knowledge, the first ones to tackle this problem were \'A. Casta\~neda and G. Robledo in \cite{JDE}, who, under some integrability conditions and without the use of spectral theory, showed that the linearization is of class $C^2$ when the linear part is a uniform contraction on $\mathbb{R}$. A similar result in this context was obtained by the same previous authors, jointly with P. Monz\'on in \cite{Monzon1}, who show that the linearization is of class $C^r \, (r \geq 1)$ when the linear part has general nonuniform contraction on $\mathbb{R}^+$.

A first result in terms of the differentiability of the linearization considering contraction and expansion, and regarding spectral theory was obtained by D. Dragi\v{c}evi\'c \textit{et al.} in \cite{Dragicevic2}; in this article, it is shown the differentiability of the linearization assuming that the linear part has a strong nonuniform exponential dichotomy, that is, the linear system has the property of nonuniform exponential dichotomy and the transition matrix has the property of nonuniform bounded growth on $\mathbb{R}^+$. Later, in \cite{CMR} the authors introduce the concept of \textit{continuous topological equivalence} which, roughly speaking, requires the continuity of the topological equivalence with respect to both variables time and space; and they prove, without spectral theory, that if the linear system has the property of uniform contraction, then the linearization is $C^r (r \geq 1)$ and the partial derivatives of this linearization are continuous with respect to both variables time and space. Recently in \cite{Jara}, N. Jara proves, without using spectral theory, that if the linear system has a general nonuniform dichotomy and the nonlinear part has restrictive assumptions; in compensation for avoiding spectral theory, then the linearization is of class $C^2$ on $\mathbb{R}^+$.

Notice that in the works previously mentioned, the global linearization considered is between a linear system and an additive perturbation of itself by a nonlinearity that is Lipschitz and bounded. Therefore, results of linearization in which it can consider unbounded nonlinearities have attracted the attention of several authors; a first result in this line is obtained by F. Lin in \cite{Lin2} who showed that if the linear system is a uniform contraction on $\mathbb{R}$ then the linearization is of class $C^0.$ In order to obtain his result, F. Lin uses the concept of \textit{almost reducibility}, \textit{i.e}, the linear system can be written as a linear diagonal system perturbed by a bounded linear term where the diagonal part is contained in the spectrum associated with the uniform hyperbolicity, and it uses the concept of crossing times with respect to the unit sphere. A second approach that deals with the problem of linearization when the nonlinearity is unbounded is carried out by I. Huerta in \cite{Huerta}, the author generalizes the previously mentioned work of F. Lin to a nonuniform framework, that is, it is constructed a linearization of class $C^0$ in two steps: the first one considers to write the linear system on $\mathbb{R}^+$ as in Lin's case where the diagonal part lies in the nonuniform spectrum, and the second step is devoted to construct a suitable Lyapunov function that plays the role of crossing time with respect to the unit sphere.

\subsection{Novelty of the article} To our knowledge, it has not yet been addressed the study about the smoothness of the linearization between a linear system with some nonautonomous hyperbolicity and a perturbation of itself by an unbounded nonlinearity. The main goal of this work is to give a first result in this direction, considering that the linear system admits uniform contraction $\mathbb{R}^+$. Namely, by one hand, avoiding the concepts of reducibility and spectral theory, we prove that the linearization of F. Lin \cite{Lin2} is of class $C^r$ ($r \geq 1$); on the other hand, we improve the work \cite{CMR} in terms that the nonlinear perturbation is unbounded.

\subsection{Structure of this article} Section 2 gives a general setting in terms of the properties and results that we use in this work. The fact that the linear system and its nonlinear perturbation are topologically equivalent is established in Section 3; additionally, in this section, we show properties that are verified for the functions that play a role in topological equivalence. Recall that the function that plays the role of topological equivalence can be written as the identity perturbed by a term that can be seen as a solution to the initial value problem; thus, in this section, we give properties of these perturbations in order that they are compatible with the definition of topological equivalence that we use in this work. In Section 4, we show that the functions associated with the topological equivalence are continuous on $(t,x) \in \mathbb{R}^+ \times \mathbb{R}^n.$ The Section 5 is devoted to the differentiability of the topological equivalence, that is, the partial derivatives up to order $r \geq 1$ of the functions that play the role of the topological equivalence are continuous on $(t,x) \in \mathbb{R}^+ \times \mathbb{R}^n.$

\subsection{Notations}
Throughout this paper, $||\cdot||$ and $|\cdot|$ will denote matrix and vector norms respectively. The set $[0,+\infty)$ is denoted by $\mathbb{R}^{+}$ and the set of square 
$n\times n$ matrices with real coefficients is denoted by $\mathcal{M}_{n}$, while $I_{n}$ is the identity matrix.

\section{Preliminary results and contextualization}

In this work we consider the following systems

    \begin{equation}\label{lin}
        x'=A(t)x
    \end{equation}
    and 
    \begin{equation}\label{nolin}
        y'=A(t)y+f(t,y).
    \end{equation}
    \vspace{0.4cm}
    
   We denote by $t\mapsto x(t,\tau,\xi)$ and $t\to y(t,\tau,\eta)$ the solutions of (\ref{lin}) and (\ref{nolin}) that pass through $\xi$ and $\eta$ respectively in $t=\tau$. We also denote by $\Phi(t,s)$ the transition matrix of (\ref{lin}) such that for $t=s$ is $I_n$. Moreover, $A:\mathbb{R}^+\to \mathcal{M}_{n}(\mathbb{R})$ is continuous and uniformly bounded, that is, there exists $M> 1$ such that
    \begin{equation}\label{acotamiento-A}
     \sup_{s \in \mathbb{R}^+}\norm{A(s)}=M < +\infty. 
    \end{equation}

    Moreover, the following properties are satisfied: 
    \vspace{0.3cm}

    \textbf{(P1)} The system (\ref{lin}) is  uniformly exponentially stable, that is, there exist constants $K\geq0$ and $\alpha>0$ such that $\Phi(t,s)$ verifies

    \begin{equation*}
        \|\Phi(t,s)\|\leq Ke^{-\alpha(t-s)} \hspace{0.3cm} \text{for any}\  t \geq s \geq 0.
    \end{equation*}
    
    \vspace{0.3cm}
    \textbf{(P2)} The function $f$ is continuous in $(t,y)$ and for any $t\geq0$ and for all $(y,\overline{y}) \in \R^n\times\R^n$ there exists $\gamma \geq 0$ such that
    
    \begin{equation*}
        |f(t,y)-f(t,\overline{y})|\leq \gamma |y-\overline{y}|.\hspace{0.3cm} 
    \end{equation*}

    \vspace{0.2cm}


      \vspace{0.2cm}

 \textbf{(P3)} There exists $\mu \geq 0$ such that $\displaystyle \sup_{t \geq 0}|f(t,0)|\leq \mu$ and $f$ is bounded on $t$ for any $x \in \R^n.$

  \vspace{0.2cm}
    
    \textbf{(P4)}  $|f(t,x)| \rightarrow +\infty$ as $|x| \rightarrow +\infty,$ for any fixed $t \geq 0.$

     \vspace{0.2cm}

    \vspace{0.2cm}
    \textbf{(P5)} The function $f(t,x)$ and its derivatives with respect to $x$ up to order $r$-th are continuous functions  of $(t,x)$.
    
    \begin{remark}
    The property $\textnormal{\textbf{(P1)}}$ says that the linear system (\ref{lin}) is contractive in the sense that the system has the uniform exponential dichotomy (nonuniform hyperbolicity) with projection $P(t) = I.$ Recall that (\ref{lin}) has the property of exponential dichotomy on $\R^+$ if there exist a projection $P(t)$, constants $K\geq 1$, $\alpha>0$  such that
\begin{displaymath}
\left\{\begin{array}{rcl}
\left \| \Phi(t,s)P(t) \right \|&\leq& K\exp(-\alpha(t-s)),\quad t\geq s, \quad t, s \in \R^+,                      \\\\
\left \| \Phi(t,s)(I-P(t)) \right \|&\leq& K\exp(\alpha(t-s)),\quad t\leq s, \quad t, s \in \R^+.
\end{array}\right.
\end{displaymath}
    
    \end{remark}

    The following result allows us to establish bounded solutions for systems of equations that rely on a parameter in an arbitrary Banach space, in particular, the solutions of the system (\ref{nolin}) are bounded. Additionally, it is analogous to those presented in \cite{CR2018, Huerta} on an uniform discrete nonautonomous framework and on a nonuniform continuous context, respectively.

\begin{proposition}\label{10}
    Consider the nonlinear perturbation
\begin{equation}\label{nonlinaux}
    x'=A(t)x+g(t,x(t),\eta)
\end{equation}
where $g:\R^+\times \R^n \times \mathcal{B}\to \R^n$ is a continuous function and $\mathcal{B}$ is a Banach space. Moreover, suppose that $g$ satisfies the following conditions:

\begin{enumerate}
    \item $g(t,x,\eta)$ is bounded with respect to $t,$ for any fixed $x \in \R^n$ and any fixed $\eta \in \B.$
    
    \item There exists $\gamma_g>0$ such that $|g(t,x_1,\eta)-g(t,x_2,\eta)|\leq \gamma_g|x_1-x_2|$, for any $t\in \R^+$ and $\eta \in \B$.
    \item $K_0=\displaystyle \sup_{t \in \R^+, \eta \in \mathcal{B}}|g(t,0,\eta)|<+\infty$.
\end{enumerate}
If $K\gamma_g<\alpha$, then for any fixed $\eta \in \B$, the system (\ref{nonlinaux}) has a unique bounded solution $X(t,\eta)$ given by

\begin{equation*}
    X(t,\eta)=\dint_0^t\Phi(t,s)g(s,X(s,\eta),\eta)ds
\end{equation*}

where \(\displaystyle \sup_{t\in \R^+}|X(t,\eta)|<+\infty\).
\end{proposition}

\begin{proof}
See \cite[Proposition 3]{Huerta}.
\end{proof}

\begin{remark}\label{para los betas}
    In particular, we can notice that all solutions of (\ref{nolin}) are bounded as a consequence of Proposition \ref{10}, \textit{i.e.}, $\displaystyle \sup_{t\in \R^+}|y(t,\tau,\eta)|<+\infty$.
\end{remark}

The next proposition and corollary are classical in the literature of differential systems in terms of local continuity with respect to the initial conditions. In \cite[Proposition 2]{Lin2} one can find the proof of these facts.

\begin{proposition}\label{1}
Let  $h:\R^+\times \R^n \to \R^n$ be a continuous function that satisfies the Lipschitz property with constant  $\gamma_h$. Then, for any $t, s \in \mathbb{R^{+}},$ the solution $y(t,s,\eta)$ of (\ref{nolin}) with $y(s,s,\eta)=\eta$ verifies 

\begin{equation*}
    \dfrac{1}{K}|\eta-\overline{\eta}|e^{(-\alpha(t-s)+K\gamma_h|t-s|)}\leq |y(t,s,\eta)-y(t,s,\overline{\eta})|\leq K|\eta-\overline{\eta}|e^{(\alpha(t-s)-K\gamma_h|t-s|)}.
\end{equation*}
\end{proposition}

\begin{corollary}\label{100}
Under the hypothesis of Proposition \ref{1} with 
$h(t,x)=A(t)x$  we have the solutions $x(t,s,\xi)$ of the system (\ref{lin})  with $x(s,s,\xi)=\xi$ satisfy

\begin{equation}\label{cor2.4}
|\xi-\overline{\xi}|e^{-M|t-s|}\leq |x(t,s,\xi)-x(t,s,\overline{\xi})|\leq |\xi-\overline{\xi}|e^{M|t-s|}.
\end{equation}

In particular, if $\overline{\xi}=0$, 

$$|\xi|e^{-M|t-s|}\leq |x(t,s,\xi)|\leq |\xi|e^{M|t-s|}.$$
\end{corollary}

\section{Topological equivalence}

 In \cite{Palmer} K.J. Palmer introduced the concept of topological equivalence between systems (\ref{lin}) and (\ref{nolin}). Roughly speaking, it is required that there exist a homeomorphism that carries solutions of linear systems into nonlinear systems and vice versa. Later, F. Lin in \cite{Lin2} presents a weaker definition of this concept of topological equivalence, which is exhibited in this work.  
    
    \begin{definition}\label{TopEq}
    The systems (\ref{lin}) and (\ref{nolin}) are $\R^{+}-$ topologically equivalent if there is a function $H:\R^+\times \mathbb{R}^n\to \mathbb{R}^n$ that satisfies
\begin{itemize}
    \item [(i)] If $x(t)$ is solution of (\ref{lin}), then $H[t,x(t)]$ is solution of (\ref{nolin});
    \item [(ii)] $H(t,x) \to H(t,x_0)$ as $x \to x_0$, uniformly with respect to $t$;
    \item[(iii)] $|H(t,x)| \to +\infty$ as $|x| \to +\infty$, uniformly with respect to $t$;
    \item [(iv)] for every fixed $\tau \in \R^+$, the map $u\mapsto H(\tau,u)$ is an homeomorphism of $\mathbb{R}^n$.
\end{itemize}
Moreover, the function $u\mapsto G(\tau,u)=H^{-1}(\tau,u)$ verifies conditions ii), iii) and iv) and maps solutions of (\ref{nolin}) on solutions of (\ref{lin}).
    \end{definition}

In \cite[Proof Theorem 1]{CMR} the authors adapted to $\mathbb{R}^+$ the construction of the homeomorphisms of the topological equivalence defined by Palmer in \cite{Palmer}. For the purpose of this work, we give some details of such a construction due to the same homeomorphisms that it will use in the context of the previous definition. Namely, assuming $K \gamma < \alpha,$ it has that, by one hand

\begin{equation}
\label{Homeo-H}
\begin{array}{rcl}
H(t,\xi)&:=&\displaystyle  \xi+\int_{0}^{t}\Phi(t,s)f(s,x(s,t,\xi)+z^{*}(s;(t,\xi)))\,ds \\\\
&=&\xi + z^{*}(t;(t,\xi)),
\end{array}
\end{equation}
where 
\begin{equation}\label{z^+}
z^{*}(t;(\tau,\xi))=\int_{0}^{t}\Phi(t,s)f(s,x(s,\tau,\xi)+z^{*}(s;(\tau,\xi))) \, ds
\end{equation}
is the unique solution of 
\begin{equation}
\label{pivote1}
\left\{\begin{array}{rcl}
z'&=&A(t)z+f(t,x(t,\tau,\xi)+z)\\
z(0)&=& 0.
\end{array}\right.
\end{equation}

Notice due to $K \gamma < \alpha$ and \textnormal{\textbf{(P1)--(P3)}}, $z^*(t;(t, \xi))$ is the fixed point of the operator $$\Gamma_{(\tau, \xi)} \colon BC(\mathbb{R}^+, \mathbb{R}^n) \to BC(\mathbb{R}^{+}, \mathbb{R}^n) ,$$ which is defined for any couple $(\tau, \xi) \in \mathbb{R}^+ \times \mathbb{R}^n$ as follows 
\begin{equation}
\phi \mapsto \Gamma_{(\tau, \xi)} \phi := \displaystyle \int_0^t \Phi(t,s) f(s,x(s,\tau,\xi) + \phi) \, ds,
\end{equation}
  where $BC(\mathbb{R}^+, \mathbb{R}^n)$ be the Banach space of bounded continuous functions with the supremum norm. Additionally, $z^*(t;(t, \xi))$ can be written as the uniform limit on $\R^+$ of the sequence $z^*_j(t;(t,\xi))$ defined recursively as follows: 

\begin{equation}\label{induccion}
\left\{\begin{array}{rcl}
z_{j+1}^{*}(t;(t,\xi))&=& \displaystyle \int_{0}^{t}\Phi(t,s)f(s,x(s,t,\xi)+z_{j}^{*}(s;(t,\xi))) \, ds \quad \textnormal{for any $j\geq 0$},\\\\
z_{0}^{*}(t;(t,\xi)) &=&    \displaystyle \int_0^t\Phi(t,s)f(s,x(s,t,\xi)) \, ds.
\end{array}\right.
\end{equation}

\begin{remark}
Recall that $z_{0}^{*}(r;(t,\xi))$ is solution of the system
\begin{equation*}
\left\{\begin{array}{rcl}
z'&=& A(t)z+f(t,x(r,t,\xi))\\ 
z(0)&=& 0,
\end{array}\right.
\end{equation*}
$z_{1}^{*}(r;(t,\xi))$ is solution of the system
\begin{equation*}
\left\{\begin{array}{rcl}
z'&=& A(t)z+f(t,x(r,t,\xi)+z_{0}^{*})\\ 
z(0)&=& 0,
\end{array}\right.
\end{equation*}
and so on; it means, for each $z_{l}^{*}(r;(t,\xi)),$ with $l=1,2, \ldots,$ is solution of the system
\begin{equation*}
\left\{\begin{array}{rcl}
z'&=& A(t)z+f(t,x(r,t,\xi)+z_{l-1}^{*})\\ 
z(0)&=& 0.
\end{array}\right.
\end{equation*}
Thus, for the purpose of the construction of the map $H,$ we focus in the solution $z_{l}^{*}(r;(t,\xi))$ at $r = t.$
\end{remark}

On the other hand,

\begin{equation}
\label{Homeo-G}
\begin{array}{rcl}
G(t,\eta)&:=&\displaystyle \eta -\int_{0}^{t}\Phi(t,s)f(s,y(s,t,\eta))\,ds \\\\
&=&\eta+w^{*}(t;(t,\eta)).
\end{array}
\end{equation}
where
\begin{equation*}
w^{*}(t;(t,\eta))=-\int_{0}^{t}\Phi(t,s)f(s,y(s,t,\eta))\,ds
\end{equation*}
is the solution of
\begin{equation}
\label{pivote2}
\left\{\begin{array}{rcl}
w'&=&A(t)w-f(t,y(t,\tau,\eta))\\
w(0)&=& 0.
\end{array}\right.
\end{equation}

Additionally, the function $H$ and $G$ have the following relation in terms of solution of (\ref{lin}) and (\ref{nolin}). 

\begin{equation*}
H[t, x(t,\tau, \xi)] = y(t,\tau, H(\tau,\xi)),
\end{equation*}
and
\begin{equation*}
G[t, y(t,\tau, \eta)] = x(t,\tau, G(\tau,\eta)) = \Phi(t, \tau)G(\tau,\eta).
\end{equation*}

\medskip

We emphasize that we will use, in the context of the Definition \ref{TopEq}, the same functions $H$ and $G$ previously mentioned, in order to establish our main results. Nevertheless, firstly, we will set some facts in terms of   $|z^*(t;(t,\xi))|$ and $|w^*(t;(t,\eta))|,$ which are considered in the construction of (\ref{Homeo-H}) and (\ref{Homeo-G}), respectively.

Additionally, throughout this section we will introduce the following property:

\begin{itemize}
    \item[\textbf{(N)}] For any $j \in \mathbb{N}$ and for any fixed $s,t \in \mathbb{R}^{+},$ 
$$|x(s,t,\xi)+z_{j}^{*}(s;(t,\xi))| \to \infty \,\, \textnormal{as} \,\, |\xi| \to \infty.$$
\end{itemize}

\bigskip

\begin{lemma}\label{lema infinito}
Under the hypothesis \textnormal{\textbf{(P1)-(P4)}}, \textnormal{\textbf{(N)}} and considering $K\gamma<\alpha,$ the solutions of the systems (\ref{pivote1}), (\ref{induccion}) and (\ref{pivote2}) satisfy, respectively:
\begin{itemize}
    \item[i)] $|z^*(t;(t,\xi))| \to +\infty$ and $|z_j^*(t;(t,\xi))| \to +\infty$ as $|\xi|\to +\infty,$ uniformly with respect to $t$.
    \item[ii)] $|w^*(t;(t,\eta))| \to +\infty$ as $|\eta|\to +\infty,$ uniformly with respect to $t$.
\end{itemize}

\end{lemma}

\begin{proof}
For i),  by Corollary \ref{100} we have that $|x(s,t,\xi)|\to +\infty$ as $|\xi| \to +\infty$ for any $t,s \in \R^+$; moreover, by \textnormal{\textbf{(P4)}}, $f(t,x)$ is unbounded on $\R^+\times \R^n$, which implies that $|z^*_0(t;(t,\xi))| \to +\infty$ as $|\xi| \to + \infty$. The conclusion for $|z^*_j(t;(t,\xi))|$ with $j \geq 1$ it is obtained due to \textnormal{\textbf{(N)}} and \textnormal{\textbf{(P4)}}.

We now show that $z^*(t;(t,\xi))$ has the same behavior of $z^*_j(t;(t,\xi))$ for any $j \in \N_0$ as $|\xi| \to +\infty.$ Suppose that there exist a sequence $\{\xi_n\}_{n \in \mathbb{N}}$ with $|\xi_n| \to \infty$ such that $$\sup_{n}{|z^*(t;(\tau,\xi_n))|}=L<+\infty.$$

From (\ref{z^+}) we have that $|f(t,x_{\xi_n})|<+\infty$ as 
$|\xi_n|\to +\infty$, where $x_{\xi_n}$ is a function such that $|x_{\xi_n}| \to +\infty$ 
as $|\xi_n| \to +\infty$. However, this a contradiction with the fact that $f(t,\cdot )$ is unbounded on $\R^n$.




In order to prove ii), we see that by using \textnormal{\textbf{(P4)}} and Proposition \ref{1}, we have that $|y(t,s,\eta)| \to +\infty$ as $|\eta| \to +\infty$ for any $t,s \in \R^+$, concluding 
$|f(s,y(s,\tau,\eta))|\to +\infty$ as $|\eta|\to +\infty$ for any $s \in \R^+$ fixed.
\end{proof}

\begin{lemma}\label{bounded} Assume that  \textnormal{\textbf{(P1)-(P3)} are satisfied}, then for any $j \in \mathbb{N}_0$, $z_j^{*}(s; (t , \xi))$ is bounded in $s$ for any $\xi \in \mathbb{R}^n.$ Moreover, we have that

$$
\begin{array}{rcl}
|z_0^{*}(s; (t , \xi))| &\leq& \dfrac{K}{\alpha}\{\gamma|x(\cdot,t,\xi)|_{\infty}+\mu\}<+\infty \, \, \textnormal{and}\\\\
|z_j^{*}(s; (t , \xi))| &\leq&  \dfrac{K}{\alpha}\{\gamma |x(s,t,\xi)+z_{j-1}^*(s,t,\xi))|_\infty +\mu \}<+\infty, \, \, \textnormal{for} \, \, j \geq 1,
\end{array}
$$
where $|x(\cdot,t,\xi)|_{\infty}=\displaystyle \sup_{s \in \R^+}|x(s,t,\xi)|.$
\end{lemma}

\begin{proof}
We know that $z_j^*$ is defined by (\ref{induccion}). So, we prove by induction over $j$ that $z_j^*$ is bounded for any $\xi \in \mathbb{R}^n.$

For $j=0$, since we have \textbf{(P1)-(P3)}, it follows that

\begin{align*}
    |z_0^*(s;(t,\xi))|&\leq \dint_0^s Ke^{-\alpha(s-r)}|f(r,x(r,t,\xi))|dr\\
    &\leq \dint_0^sKe^{-\alpha(s-r)}\{ \gamma|x(r,t,\xi)|+\mu\}dr\\
    &\leq \dint_0^s K\gamma e^{-\alpha(s-r)}|x(r,t,\xi)|dr+\dint_0^sKe^{-\alpha(s-r)}\mu dr\\
    &\leq \dint_0^t  K\gamma e^{-\alpha(s-r)}|x(r,t,\xi)|dr+ \dfrac{K\mu}{\alpha}.\\
\end{align*}
Since $x(s,t,\xi)$ is solution of (\ref{lin}), thus is bounded in $s \in \R^+$, therefore $|x(\cdot,t,\xi)|_{\infty}<+\infty$, which implies that

\begin{align*}
    |z_0^*(s;(t,\xi))|&\leq  K\gamma |x(\cdot,t,\xi)|_{\infty}\dint_0^se^{-\alpha(s-r)}dr+\dfrac{K\mu}{\alpha}\\
    &\leq \dfrac{K}{\alpha}\{\gamma|x(\cdot,t,\xi)|_{\infty}+\mu\}<+\infty.
\end{align*}

Now, we will assume  the inductive hypothesis, that is, for some $n \in \mathbb{N}$ we have that  $|z^*_n(s;(t,\xi))|$ is bounded in $s \in \mathbb{R}^n.$ For the step $n+1$, since we have \textbf{(P1)-(P3)}, it follows that

\begin{align*}
    |z_{n+1}^*(s;(t,\xi))|&\leq \dint_0^s Ke^{-\alpha(s-r)}|f(r,x(r,t,\xi)+z_n^*(r;(t,\xi))|dr\\
    &\leq \dint_0^s Ke^{-\alpha(s-r)}\{\gamma|x(r,t,\xi)+z_n^*(r;(t,\xi))|+\mu \}dr.
\end{align*}

By induction hypothesis $|z_n^*(\cdot;(t,\xi))|_\infty<+\infty$, thus $|x(\cdot,t,\xi)+z_n^*(\cdot;(t,\xi))|_{\infty}<+\infty$. 

Finally, we obtain the following

\begin{align*}
    |z_{n+1}^*(s;(t,\xi))|&\leq \dfrac{K}{\alpha}\{\gamma |x(\cdot,t,\xi)+z_n^*(\cdot,t,\xi))|_\infty +\mu \}<+\infty.
\end{align*}

Therefore, the result follows.
\end{proof}

\begin{lemma}\label{zj uc}
Assume that \textbf{(P1)-(P2)} are satisfied, then for any $j \in \mathbb{N}_0$, $z_j^{*}(t; (t , \xi))$ is uniformly continuous with respect to $\xi.$
\end{lemma}

\begin{proof}
The proof will be done by induction.  Before that, we introduce the following auxiliary function $\theta_0(t):[0,+\infty) \to [0,+\infty)$ defined by:

\begin{equation*}
\theta_{0}(t)=\left\{\begin{array}{lcl}
K\gamma t & \textnormal{if} & \alpha=M,\\\\
K\gamma\left(\frac{e^{(M-\alpha)t}-1}{M-\alpha}\right)  &\textnormal{if} & \alpha<M.
\end{array}\right.
\end{equation*}

Now, given  $\varepsilon>0$, let us define the constants

\begin{equation*}
    L^*(\varepsilon)=\dfrac{1}{\alpha}\ln\left(\dfrac{2\gamma \omega K}{\alpha \varepsilon}\right)\hspace{0.3cm} \text{and} \hspace{0.3cm}  \Theta_0^*=\displaystyle \max_{t \in [0,L^*(\varepsilon)]}\theta_0(t)
\end{equation*}

where $\omega=\omega_\xi+ \omega_{\overline{\xi}}$ with $\omega_\xi=\displaystyle \sup_{s \in \R^+} |x(s,t,\xi)+z_j^*(s;(t,\xi))|$.


We will distinguish the cases $t\in [0,L^*(\varepsilon)]$ and $t>L^*(\varepsilon)$ and we will use the following notation

\begin{equation*}
    \Delta_j(t,\xi,\overline{\xi})=z^*_j(t;(t,\xi))-z^*_j(t;(t,\overline{\xi})).
\end{equation*}

First, for $j=0$ and $t\in [0,L^*(\varepsilon)],$ by using (\ref{acotamiento-A}) combined with \textbf{(P1)-(P2)} and the right inequality in (\ref{cor2.4}), we can verify that

\begin{align*}      
|\Delta_0(t,\xi,\overline{\xi})|&\leq K\gamma e^{-\alpha t}\dint_0^t  e^{\alpha s}|x(s,t,\xi)-x(s,t,\overline{\xi})|ds\\
    &\leq K\gamma e^{-\alpha t}\dint_0^t e^{\alpha s}|\xi-\overline{\xi}|e^{M(t-s)}ds\\
    &\leq K\gamma \left \{\dfrac{e^{(M-\alpha)t}-1}{M-\alpha}\right \}|\xi-\overline{\xi}|\\
    &\leq \Theta_0^*|\xi-\overline{\xi}|.
\end{align*}

On the other hand, when $t>L^*(\varepsilon)$, \textbf{(P1)-(P2)} we have

\begin{align*}
    |\Delta_0(t,\xi,\overline{\xi})|&\leq K \gamma e^{-\alpha t}\dint_0^{t-L^*}e^{\alpha s}\{|x(s,t,\xi)|+|x(s,t,\overline{\xi})|\}ds + K \gamma \dint_{t-L^*}^t e^{-\alpha(t- s)} \{|x(s,t,\xi)-x(s,t,\overline{\xi})|\}ds\\
    &\leq  K \gamma e^{-\alpha t}\dint_0^{t-L^*}e^{\alpha s} \omega ds + K \gamma e^{-\alpha t}\dint_{t-L^*}^t e^{\alpha s} \{|x(s,t,\xi)-x(s,t,\overline{\xi})|\}ds.
\end{align*}

By (\ref{acotamiento-A}) combined with $u=t-s$ and the right inequality of (\ref{cor2.4}), it follows that

\begin{align*}
    |\Delta_0(t,\xi,\overline{\xi})|&\leq \dfrac{K \gamma \omega}{\alpha}e^{-\alpha L^*}+ K\gamma \dint_0^{L^*} e^{-\alpha u}\{|x(t-u,t,\xi)-x(t-u,t,\overline{\xi})|\}du\\
    &\leq \dfrac{K \gamma \omega}{\alpha}e^{-\alpha L^*}+ K\gamma \dint_0^{L^*} |\xi-\overline{\xi}|e^{(M-\alpha) u}du\\
    &\leq \dfrac{\varepsilon}{2} +K\gamma |\xi-\overline{\xi}|\left\{ \dfrac{e^{(M-\alpha)L^*}-1}{M-\alpha}\right\}\\
    &\leq \dfrac{\varepsilon}{2}+ \Theta_0^*|\xi-\overline{\xi}|.
\end{align*}

Secondly, we will assume the inductive hypothesis
\begin{equation*}
    \forall \varepsilon>0, \ \exists \ \delta_j(\varepsilon)>0 \text{ s.t. } |\xi-\overline{\xi}|<\delta_j(\varepsilon) \implies |z_j^*(t;(t,\xi))-z_j^*(t;(t,\overline{\xi}))|<\varepsilon \hspace{0.3cm} \text{ for any}\  t \geq 0.
\end{equation*}

For the step $j+1$, as before, we will distinguish the cases $t \in [0,L^*(\varepsilon)]$ and $t >L^*(\varepsilon)$.
Then, when $t \in [0,L^*(\varepsilon)]$ and for a given $\varepsilon>0$, using the right inequality of (\ref{cor2.4}) and \textbf{(P1)-(P2)}, we can verify

\begin{align*}
    |\Delta_{j+1}(t,\xi,\overline{\xi})|&\leq K\gamma e^{-\alpha t}\dint_0^te^{\alpha s}\{|x(s,t,\xi)-x(s,t,\overline{\xi})|+|\Delta_j(s,\xi,\overline{\xi})|\}ds\\
    &\leq K\gamma e^{-\alpha t}\dint_0^te^{\alpha s}\{|\xi-\overline{\xi}|e^{M(t-s)}+|\Delta_j(\cdot,\xi,\overline{\xi})|_{\infty}\}ds\\
    &\leq K\gamma\left\{ \dfrac{e^{(M-\alpha)t}-1}{M-\alpha}\right \}|\xi-\overline{\xi}|+ \dfrac{K\gamma}{\alpha}|\Delta_j(\cdot,\xi,\overline{\xi})|_{\infty}\\
    &\leq \Theta^*_0 |\xi-\overline{\xi}|+\dfrac{K\gamma}{\alpha}|\Delta_j(\cdot,\xi,\overline{\xi})|_{\infty}
\end{align*}
where $|\Delta_j(\cdot,\xi,\overline{\xi})|_{\infty}=\displaystyle \sup_{t\geq 0}|\Delta_j(t,\xi,\overline{\xi})|$.

 When  $t> L^*(\varepsilon)$ we use (\ref{acotamiento-A}), and the Lipschitzness  of  $f$ in order to deduce that

\begin{align*}
    |\Delta_{j+1}(t,\xi,\overline{\xi})| &\leq K\gamma \dint_0^{t-L^*}e^{-\alpha(t-s)}\{|x(s,t,\xi)+z^*_j(s;(t,\xi))|+|x(s,t,\overline{\xi})+z^*_j(s;(t,\overline{\xi}))|\}ds\\
    &\hspace{0.7cm}+K\gamma\dint^t_{t-L^*} e^{-\alpha(t-s)}\{|x(s,t,\xi)-x(s,t,\overline{\xi})|+|\Delta_j(s,\xi,\overline{\xi})|\}ds.
\end{align*}

By Lemma \ref{bounded} combined with $u=t-s$ and the right inequality in (\ref{cor2.4}), we have that
\begin{align*}
    |\Delta_{j+1}(t,\xi,\overline{\xi})|&\leq K\gamma \dint_0^{t-L^*}e^{-\alpha(t-s)}\omega ds+K\gamma \dint_0^{L^*}e^{-\alpha u}|x(t-u,t,\xi)-x(t-u,t,\overline{\xi})|du\\
    &\hspace{1.0cm}+ K\gamma \dint_0^{L^*}e^{-\alpha u}|\Delta_j(\cdot,\xi,\overline{\xi})|_{\infty}du\\
    &\leq \dfrac{K\gamma \omega}{\alpha}(e^{-\alpha L^*}-e^{-\alpha t})+K\gamma|\xi-\overline{\xi}|\dint_0^{L^*}e^{(M-\alpha)u}du+\dfrac{K \gamma}{\alpha}(1-e^{-\alpha L^*})|\Delta_j(\cdot,\xi,\overline{\xi})|_{\infty}\\
    &\leq \dfrac{K\gamma \omega}{\alpha}e^{-\alpha L^*}+K\gamma|\xi-\overline{\xi}|\left\{\dfrac{e^{(M-\alpha)L^*}-1}{M-\alpha} \right \} + \dfrac{K \gamma}{\alpha}|\Delta_j(\cdot,\xi,\overline{\xi})|_{\infty}\\
    &\leq \dfrac{\varepsilon}{2}+\Theta^*_0|\xi-\overline{\xi}|+\dfrac{K\gamma}{\alpha}|\Delta_j(\cdot,\xi,\overline{\xi})|_{\infty}.
\end{align*}

Summarizing, for any $t\geq 0$ it follows that

\begin{displaymath}
|\Delta_{j+1}(t,\xi,\bar{\xi})|\leq 
\left\{\begin{array}{lcl}
\displaystyle 
 \Theta_{0}^{*}|\xi-\bar{\xi}|+\frac{K\gamma}{\alpha} |\Delta_{j}(\cdot,\xi,\bar{\xi})|_{\infty}& \textnormal{if}& t\in [0,L^*]\\\\ \dfrac{\varepsilon}{2}+\Theta^*_0|\xi-\overline{\xi}|+\dfrac{K\gamma}{\alpha}|\Delta_j(\cdot,\xi,\overline{\xi})|_{\infty} &\textnormal{if}& t>L^*.
\end{array}\right.
\end{displaymath}
Thus, taking 

$$\delta_{j+1}(\varepsilon)=\min \left\{\delta_j(\varepsilon/2), \dfrac{\varepsilon}{2\Theta^*_0} \left(1-\dfrac{K\gamma}{\alpha}\right) \right \}$$
for any $t\geq 0$, we have that

\begin{equation*}
    \forall  \varepsilon>0, \exists \ \delta_{j+1}(\varepsilon)>0 \  \text{s.t.} \ |\xi-\x|<\delta_{j+1} \implies |z^*_{j+1}(t;(t,\xi))-z^*_{j+1}(t;(t,\x))|<\varepsilon,
\end{equation*}
and the uniform continuity of $\xi \mapsto z_{j}^{*}(t;(t,\xi)) $ follows for any $j\in \mathbb{N}$.

\end{proof}

As we have set forth the premises, we are now able to state our main result of this section, which establishes that the systems (\ref{lin}) and (\ref{nolin}) are topologically equivalent on $\mathbb{R}^+.$

\begin{theorem}
\label{teorema1}
Assume that {\bf{(P1)--\bf{(P4)}}},  {\bf{(N)}} are satisfied and $K \gamma < \alpha,$
then systems \textnormal{(\ref{lin})} and \textnormal{(\ref{nolin})} are $\mathbb{R}^+-$ topologically equivalent.
\end{theorem}

\begin{proof} 
Firstly, recall that the maps $H$ and $G$ are defined in (\ref{Homeo-H}) and (\ref{Homeo-G}), respectively.

Secondly, the item i) of Definition \ref{TopEq} follows by \cite[Proof Theorem 1]{CMR}.
In similar way, the proof of iv) follows considering \cite[Step 4 Proof Theorem 1]{CMR} which establish the bijectiveness of $H$ and $G.$

In order to prove iii) of Definition \ref{TopEq}, we show that  
 $|H(t,\xi)|\to +\infty$ as $|\xi|\to +\infty$. By item i) of Lemma \ref{lema infinito} we have that this fact is verified for $z^*(t;(t,\xi)).$ 
Therefore, we can conclude that $|H(t,\xi)|\to +\infty$ as $|\xi|\to +\infty$. Similarly, it's  showed that $|G(t,\eta)|\to +\infty$ as $|\eta | \to +\infty$  due to item ii) Lemma \ref{lema infinito}.

Finally we show ii), that is,  $H$ and $G$ are uniformly continuous with respect to $\xi$ for any $t \geq 0$.  We construct the auxiliary function $\theta:[0,+\infty) \to [0,+\infty)$ defined by

\begin{displaymath}
\theta(t)=1+ K\gamma \left(\frac{e^{(M+\gamma-\alpha)t}-1}{M+\gamma-\alpha}\right) 
\end{displaymath}

Now, given $\varepsilon>0$, let us define the following constants

\begin{equation*}
    L(\varepsilon)=\dfrac{1}{\alpha}\ln\left(\dfrac{2\gamma \beta K}{\alpha \varepsilon}\right) \hspace{0.3cm}  \text{and} \hspace{0.3cm} \theta^*=\displaystyle \max_{t \in [0,L(\varepsilon)]}\theta(t),
\end{equation*}
where  $\beta=\beta_1 + \beta_2$ with $\beta_1=\displaystyle \sup_{t \in \R^+}|y(t,\tau,\eta)|, \, \beta_2=\displaystyle \sup_{t \in \R^+}|y(t,\tau,\n)|;$ which are well defined by Remark \ref{para los betas}.

We will prove the uniform continuity of $G$ by considering two cases: 

\noindent\emph{Case i)}  $t \in [0,L(\varepsilon)]$. By {\bf{(P1)}} and {\bf{(P2)}} we can deduce that
    \begin{equation}\label{300}
        |G(t,\eta)-G(t,\n)|\leq |\eta-\n|+K\gamma e^{-\alpha t}\dint_0^te^{\alpha s}|y(s,t,\eta)-y(s,t,\n)|ds
    \end{equation}

By {\bf{(P2)}} and (\ref{acotamiento-A}) we obtain for any $0\leq s \leq t$:

$$
\begin{array}{rcl}
    |y(s,t,\eta)-y(s,t,\n)|&\leq &  |\eta-\n|+\dint_s^t\|A(\tau)\| |y(\tau,t,\eta)-y(\tau,t,\n)|d\tau \\ 
    & & +\dint_s^t|f(\tau,y(\tau,t,\eta))-f(\tau,y(\tau,t,\n))|d\tau  \\
    &\leq& |\eta-\n|+(M+\gamma)\dint_s^t|y(\tau,t,\eta)-y(\tau,t,\n)|d\tau.
\end{array}
$$

By using Gronwall's Lemma, we conclude that for any $0\leq s\leq t:$

\begin{equation}
\label{400}
    |y(s,t,\eta)-y(s,t,\n)| \leq |\eta-\n|e^{(M+\gamma)(t-s)}.
\end{equation}

Upon inserting (\ref{400}) in (\ref{300}), we obtain that

\begin{align*}
    |G(t,\eta)-G(t,\n)|\leq & \left(1+K\gamma e^{(M+\gamma-\alpha)t}\dint_0^te^{-(M+\gamma-\alpha)s}\right)|\eta-\n| \\
    =&\left(1+K\gamma \left \{\dfrac{e^{(M+\gamma-\alpha)t}-1}{M+\gamma-\alpha} \right \} \right)|\eta-\n| \\
    \leq& \  \theta(t)|\eta-\n| \\
    \leq& \  \theta^*|\eta-\n|.
\end{align*}

\noindent\emph{Case ii)}  $t>L(\varepsilon).$ By \textbf{(P1)-(P2)} we have,

\begin{align*}
    |G(t,\eta)-G(t,\n)|\leq& \  |\eta-\n|+\dint_0^{t-L}Ke^{-\alpha(t-s)}|f(s,y(s,t,\eta))-f(s,y(s,t,\n))|ds\\
    &\hspace{0.2cm}+ \dint_{t-L}^{t}Ke^{-\alpha(t-s)}|f(s,y(s,t,\eta))-f(s,y(s,t,\n))|ds\\
    \leq& \ |\eta-\n| + \dint_0^{t-L}K\gamma e^{-\alpha(t-s)}\{|y(s,t,\eta)|+|y(s,t,\n)|\}ds\\
    &\hspace{0.2cm}+ \dint_{t-L}^{t}K\gamma e^{-\alpha(t-s)}|y(s,t,\eta)-y(s,t,\n)|ds\\
    \leq& \ |\eta-\n| + K\gamma \beta \dint_0^{t-L}e^{-\alpha(t-s)}ds+\dint_{t-L}^{t}K\gamma e^{-\alpha(t-s)}|y(s,t,\eta)-y(s,t,\n)|ds\\
    =&|\eta-\n| + K\gamma \beta \dint_0^{t-L}e^{-\alpha(t-s)}ds+\dint_{0}^{L}K\gamma e^{-\alpha u}|y(t-u,t,\eta)-y(t-u,t,\n)|du.
\end{align*}

As in case i), the inequality (\ref{400}) implies
\begin{align*}
    K\gamma \dint_0^Le^{-\alpha u}|y(t-u,t,\eta)-y(t-u,t,\n)|du&\leq K\gamma\dint_0^Le^{(M+\gamma-\alpha)u}|\eta-\n|du\\
    &=K\gamma \left\{\dfrac{e^{(M+\gamma-\alpha)L}-1}{M+\gamma-\alpha} \right\}|\eta-\n|.
\end{align*}

Therefore,

\begin{align*}
    |G(t,\eta)-G(t,\n)|&\leq \left(1+K\gamma\left\{\dfrac{e^{(M+\gamma-\alpha)L}-1}{M+\gamma-\alpha}  \right\}\right)|\eta-\n|+\dfrac{K\gamma\beta}{\alpha}e^{-\alpha L}\\
    &\leq \theta^*|\eta-\n|+\dfrac{\varepsilon}{2}.
\end{align*}

Summarizing, given  $\varepsilon>0$, there exists $L(\varepsilon)>0$ and $\theta^*$ such that:

\begin{displaymath}
|G(t,\eta)-G(t,\bar{\eta})|\leq \left\{\begin{array}{lcl}
\displaystyle \theta^{*}|\eta-\bar{\eta}| &\textnormal{if}& t\in [0,L]\\\\
\displaystyle \theta^{*}|\eta-\bar{\eta}|+\frac{\varepsilon}{2} &\textnormal{if}& t>L.
\end{array}\right.
\end{displaymath}

Then it follows that
$$
\forall \varepsilon>0 \,\exists
\delta(\varepsilon)=\frac{\varepsilon}{2\theta^{*}}\quad \textnormal{such that} \quad
|\eta-\bar{\eta}|<\delta \Rightarrow |G(t,\eta)-G(t,\bar{\eta})|<\varepsilon
$$
and the uniform continuity of $G$ follows.

Now we will prove that $H$ is uniformly continuous for any $t\geq 0$, and as the identity is uniformly continuous, we only need to prove that $\xi \mapsto z^*(t;(t,\xi))$ is uniformly continuous.

We noticed before that the fixed point $z^*(t;(t,\xi))$ can be written as the uniform limit on $\R^+$ of the sequence $z_j^*(t;(t,\xi))$ defined in (\ref{induccion}), so the uniform continuity of each map $\xi \mapsto z_j^*(t;(t,\xi))$  follows from Lemma \ref{zj uc}.

In order to finish our proof, we choose $N\in \mathbb{N}$ such that for any $j>N$ fixed, it follows that
\begin{displaymath}
|z^{*}(\cdot;(\cdot,\xi))-z_{j}^{*}(\cdot;(\cdot,\xi))|_{\infty}<\varepsilon \quad \textnormal{for any} \quad  \xi\in \mathbb{R}^{n},
\end{displaymath} 
and therefore, if $|\xi-\bar{\xi}|<\delta_{j}$ with $j>N$, it is true that
\begin{displaymath}
\begin{array}{rcl}
|z^{*}(t;(t,\xi))-z^{*}(t;(t,\bar{\xi}))|&\leq & |z^{*}(t;(t,\xi))-z_{j}^{*}(t;(t,\xi))|+\Delta_{j}(t,\xi,\bar{\xi})\\\\
&&+|z^{*}(t;(t,\bar{\xi}))-z_{j}^{*}(t;(t,\bar{\xi}))|<3\varepsilon,
\end{array}
\end{displaymath}
and the uniform continuity of $\xi\mapsto z^{*}(t;(t,\xi))$ and $\xi\mapsto H(t,\xi)$ follows for any fixed $t\geq 0$.

\end{proof}

\begin{remark}
The proof of Theorem \ref{teorema1} follows the ideas of \cite[Theorem 2.1]{CMR}. This last result, on the one hand,
considers a definition of $\mathbb{R}^+-$topological equivalence stronger than Definition \ref{TopEq}, however, on the other hand, and compensatory, it is used that
$\abs{f(t,y)} \leq \mu$ for any $t \geq 0$, while in our case we only have that $\abs{f(t,0)} \leq \mu$ for any $t \geq 0.$ These facts establish a contrast between both perspectives, highlighting
that the proof of Theorem \ref{teorema1} have more subtle technical details  than \cite[Theorem 2.1]{CMR}. Indeed, the role of unboundness of $(t,x) \mapsto f(t, x)$ is hard to handle to show that $|H(t,\xi)|\to +\infty$ as $|\xi|\to +\infty$ and that $|G(t,\eta)|\to +\infty$ as $|\eta|\to +\infty.$
 \end{remark}
 
 \section{Continuity of Topological Equivalence}
 
 In this section, we deal with the fact that the maps $H$ and $G$ constructed in the previous section are both continuous on $\mathbb{R}^+ \times \mathbb{R}^n.$ In order to formalize this fact, we recall the definition of $\mathbb{R}^{+}$-- continuously topologically equivalent introduced in \cite{CMR}, however tailored to the context of this work.
 
 \begin{definition} The systems \textnormal{(\ref{lin})} and \textnormal{(\ref{nolin})} are $\mathbb{R}^+$--continuously topologically equivalent if 
there exists a function $H\colon \mathbb{R}^+ \times \mathbb{R}^{n}\to \mathbb{R}^{n}$ with the properties
 \begin{itemize}
\item[(i)] If $x(t)$ is a solution of \textnormal{(\ref{lin})}, then $H[t,x(t)]$ is a solution 
of \textnormal{(\ref{nolin})};
\item[(ii)]$H(t,x) \to H(t,x_0)$ as $x \to x_0$, uniformly with respect to $t$;
    \item[(iii)] $|H(t,x)| \to +\infty$ as $|x| \to +\infty$, uniformly with respect to $t$;
\item[(iv)]  for each fixed $t\in \mathbb{R}^+$, $u\mapsto H(t,u)$ is an homeomorphism of $\mathbb{R}^{n}$;
\item[(v)] $H$ is continuous in any $(t,u)\in \mathbb{R}^+\times \mathbb{R}^{n}$.
\end{itemize}
In addition, the function $G(t,u)=H^{-1}(t,u)$ has properties \textnormal{(ii)--(v)} and maps solutions of \textnormal{(\ref{nolin})} into solutions of \textnormal{(\ref{lin})}.
 \end{definition}
 
 The following result, the main of this section, states that the functions $H(t,x)$ and $G(t,x)$ given by Theorem \ref{teorema1}, satisfy that are continuous functions on $(t,x) \in \mathbb{R}^+\times \mathbb{R}^{n}.$

 \begin{theorem}
 Assume that {\bf{(P1)--\bf{(P4)}}}, {\bf{(N)}} are satisfied and $K\gamma  < \alpha$, then the systems (\ref{lin}) and (\ref{nolin}) are $\R^+$--continuously topologically equivalent.
 \end{theorem}
 
 \begin{proof}
 By Theorem \ref{teorema1}, we know that the systems (\ref{lin}) and (\ref{nolin}) are $\R^+$--topologically equivalent. Moreover, by using the continuity of the solutions with respect to the initial time and initial conditions  \cite[Ch.V]{Hartman}, we note that for any $\varepsilon_{1}>0$ there exists $\delta_{1}(t_{0},\xi_{0},\varepsilon_{1})>0$ such that
\begin{equation}
\label{crt1}
|x(s,t,\xi)-x(s,t_{0},\xi_{0})|< \varepsilon_{1}  \quad\textnormal{whenever} \quad |t-t_{0}|+|\xi-\xi_{0}|<\delta_{1}  
\end{equation}
for any $t$ and $s$.
\vspace{0.2cm}

On the other hand, by using the continuity of the solutions with respect of the parameters \cite[Ch.V]{Hartman} combined with Proposition \ref{10}, we know that for any $\varepsilon_2>0$, there exist $\delta_2(t_0,\xi_0,\varepsilon_2)>0$ such that

\begin{equation}\label{crt2}
    |z^*(s;(t,\xi))-z^*(s;(t_0,\xi_0))|<\varepsilon_2 \hspace{0.3cm} \text{whenever} \hspace{0.3cm} |t-t_0|+|\xi-\xi_0|<\delta_2
\end{equation}
for any $t$ and $s.$ 

Additionally, we know for any $\varepsilon_3>0$ there exists $\delta_3(\varepsilon_3,t_0)$ such that

\begin{equation}\label{crt3}
    \|\Phi(t,s)-\Phi(t_0,s)\|<\varepsilon_3 \hspace{0.3cm} \text{whenever} \hspace{0.3cm} |t-t_0|<\delta_3
\end{equation}
for any $t$ and $s.$ 

From now on, we will assume that $t$,$s$ and $t_{0}$ are in a compact interval $I\subset \mathbb{R}^{+}$ and we denote
\begin{displaymath}
\omega_{1}(s,t,t_{0},\xi,\xi_{0})=f(s,x(s,t,\xi)+z^{*}(s;(t,\xi)))-f(s,x(s,t_{0},\xi_{0})+z^{*}(s;(t_{0},\xi_{0}))).
\end{displaymath}

We will assume that $t>t_{0}$.  Now, we can notice that

\begin{align*}
    H(t,\xi)-H(t_0,\xi_0)=\ &\xi-\xi_0+\dint_0^t\Phi(t,s)f(s,x(s,t,\xi)+z^*(s;(t,\xi)))ds\\
    &\hspace{0.2cm}-\dint_0^{t_0}\Phi(t_0,s)f(s,x(s,t_0,\xi_0)+z^*(s;(t_0,\xi_0)))ds\\
    =\ &\xi-\xi_0 +\dint_0^{t_0}\{\Phi(t,s)-\Phi(t_0,s)\}f(s,x(s,t,\xi)+z^*(s;(t,\xi)))ds\\
    &\hspace{0.2cm} +\dint_0^{t_0}\Phi(t_0,s)\omega_1(s,t,t_0,\xi,\xi_0)ds\\
    &\hspace{0.2cm}+\dint_{t_0}^{t}\Phi(t,s)f(s,x(s,t,\xi)+z^*(s;(t,\xi)))ds.
\end{align*}

Let $C=\max\{\|\Phi(u,s)\|: u,s \in I\}$.  By using \textbf{(P3)} and the fact that $f(t,x)$ is continuous in $(t,x) \in \R^+\times \R^n$ 
, we have that $f(s,x(s,t,\xi)+z^*(s;(t,\xi)))$ is bounded on $I$. Thus, there exists $\rho \geq 0$ such that

\begin{equation}\label{crt4}
    |f(s,x(s,t,\xi)+z^*(s;(t,\xi)))|\leq \rho \  \text{for any} \ t,s \in I.
\end{equation}

By using (\ref{crt1}), (\ref{crt2}), (\ref{crt3}) and (\ref{crt4}), we have that

\begin{align*}
    |H(t,\xi)-H(t_0,\xi_0)|&\leq |\xi-\xi_0|+\rho\dint_0^{t_0}\|\Phi(t,s)-\Phi(t_0,s)\|ds + \rho\dint_{t_0}^t\|\Phi(t,s)\|ds\\
    &\hspace{0.2cm} +\dint_0^{t_0}\|\Phi(t_0,s)\| |\omega_1(s,t,t_0,\xi,\xi_0)|ds\\
    &\leq |\xi-\xi_0|+\rho t_0\varepsilon_3+C|t-t_0|\rho +C \gamma(\varepsilon_1+\varepsilon_2)t_0.
\end{align*}

The case that $t<t_0$ is analogous to previous analysis. Therefore, we conclude that $H$ is continuous in any $(t_0,\xi_0)\in \R^+ \times \R^n$.

Now, we show that $G$ is continuous in any $(t_0,\eta_0) \in \R^+ \times \R^n$. By using the continuity of the solutions with respect to parameters together with Proposition \ref{10} we know that for any $\varepsilon_4 > 0$ there exists $\delta_4(t_0,\xi_0,\varepsilon_4)>0$ such that

\begin{equation}\label{crt5}
    |y(s,t,\eta)-y(s,t_0,\eta_0)|<\varepsilon_4 \hspace{0.3cm} \text{when} \hspace{0.3cm} |t-t_0|+|\eta-\eta_0|<\delta_4.
\end{equation}

In the case $t_0<t,$ we obtain

\begin{align*}
    G(t,\eta)-G(t_0,\eta_0)=&\eta-\eta_0-\dint_0^t\Phi(t,s)f(s,y(s,t,\eta))ds+\dint_0^{t_0}\Phi(t_0,s)f(s,y(s,t_0,\eta_0))ds\\
    =&\eta-\eta_0-\dint_0^{t_0}\Phi(t,s)f(s,y(s,t,\eta))ds-\dint_0^{t_0}\Phi(t_0,s)f(s,y(s,t,\eta))ds\\
    \hspace{0.5cm}&+\dint_0^{t_0}\Phi(t_0,s)f(s,y(s,t,\eta))ds+
    \dint_0^{t_0}\Phi(t_0,s)f(s,y(s,t_0,\eta_0))ds\\
    \hspace{0.5cm}&-\dint_{t_0}^t\Phi(t,s)f(s,y(s,t,\eta))ds\\
    =& \eta-\eta_0+\dint_0^{t_0}\{\Phi(t_0,s)-\Phi(t,s)\}f(s,y(s,t,\eta))ds-
    \dint_{t_0}^t\Phi(t,s)f(s,y(s,t,\eta))ds\\
    \hspace{0.3cm}&+\dint_0^{t_0}\Phi(t_0,s)\{f(s,y(s,t_0,\eta_0))-
    f(s,y(s,t,\eta))\}ds,
\end{align*}
and by using (\ref{crt3}), (\ref{crt4}) and (\ref{crt5}) we have that

\begin{align*}
    |G(t,\eta)-G(t_0,\eta_0)|\leq& |\eta-\eta_0|+\rho\dint_0^{t_0}\|\Phi(t,s)-\Phi(t_0,s)\|ds+\rho\dint_{t_0}^t\|\Phi(t,s)\|ds\\
    \hspace{0.5cm}&+\dint_0^{t_0}\|\Phi(t_0,s)\| |f(s,y(s,t,\eta))-
    f(s,y(s,t_0,\eta_0))|ds\\
    \leq&|\eta-\eta_0|+\rho t_0\varepsilon_3+\rho|t-t_0|C+C\gamma\varepsilon_4.
\end{align*}

The case $t < t_0$ is similar. Therefore, we conclude that $G$ is continuous in any $(t_0,\eta_0) \in \R^+ \times \R^n.$ 

 \end{proof}
 
 \begin{remark}
 In Theorem 2.3 of \cite{CMR} is obtained a similar bound for  $|H(t,\xi)-H(t_0,\xi_0)|$ which depends on the bound of $\abs{f(t,x)}.$ As this hypothesis is dropped in this work, we have that the bound for this estimation depends on the boundness of the $\abs{f(s,x(s,t,\xi)+z^*(s;(t,\xi)))}$ for any $t \in I.$
 \end{remark}
 
  \section{Differentiability of Topological Equivalence}
  
 As before, we recall Definition \cite[Definition 1.5]{CMR} and we adapt it to the context of this work.
  
 \begin{definition}
\label{TopEqCr}
The systems \textnormal{(\ref{lin})} and \textnormal{(\ref{nolin})} are $C^{r}-$ continuously topologically equivalent on $\mathbb{R}^+$ if: 
\begin{itemize}
\item[(i)] The systems are $\mathbb{R}^+-$continuously topologically equivalent;
\item[(ii)] for any fixed $t\in \mathbb{R}^+$; the map $u\mapsto H(t,u)$ is a $C^{r}$--diffeomorphism of $\mathbb{R}^{n}$;
with $r\geq 1$,
\item[(iii)] the partial derivatives of $H$ and $G$ up to order $r$ with respect to $u$ are continuous functions of $(t,u)\in \mathbb{R}^+\times \mathbb{R}^{n}$. 
\end{itemize}
\end{definition}

The following is the main result of this work. We prove that the topological equivalence, established in Theorem \ref{teorema1} is of class $C^r, r \geq 1$ on the half line. 
 
 \begin{theorem}\label{teorema 3} Assume that \textbf{(P1)-(P5)}, {\bf{(N)}} and $K \gamma<\alpha$ are verified, then (\ref{lin}) and (\ref{nolin}) are $C^r-$ continuous topologically equivalent on $\R^+$.
 \end{theorem}
 
 \begin{proof} The property (i) of Definition \ref{TopEqCr} is satisfied by Theorem \ref{teorema1}.  The property (ii) of Definition previously mentioned will be established by cases.
 
 \medskip
 
 \noindent{\bf{Case $r=1$.}} 
 
The expression for the first partial derivatives of the map $\eta\mapsto G(t,\eta)$ for any $t\geq 0$ are the following
\begin{equation}
\label{derivada-parcial}
\frac{\partial G}{\partial \eta_{i}}(t,\eta)=e_{i}-\int_{0}^{t}\Phi(t,s)Df(s,y(s,t,\eta))\frac{\partial y}{\partial \eta_{i}}(s,t,\eta)\,ds \quad (i=1,\ldots,n), 
\end{equation}
where $Df$ is the Jacobian matrix of $f,$ which implies that the partial derivatives exists and are continuous for any fixed $t\geq 0$, then $\eta \mapsto G(t,\eta)$ is $C^{1}$.

On the other hand, by using (\textbf{P5}) we have that
$\partial y(t,\tau,\eta)/\partial \eta$ satisfies the matrix differential equation
\begin{equation} 
\label{MDE1}
\left\{
\begin{array}{rcl}
\displaystyle \frac{d}{dt}\frac{\partial y}{\partial\eta}(t,\tau,\eta)&=&\displaystyle \{A(t)+Df(t,y(t,\tau,\eta))\}\frac{\partial y}{\partial \eta}(t,\tau,\eta),\\\\
\displaystyle \frac{\partial  y}{\partial\eta}(\tau,\tau,\eta)&=&I_n.
\end{array}\right.
\end{equation}

Now, by using the identity $\Phi(t,s)A(s)=-\frac{\partial }{\partial s}\Phi(t,s)$ combined with
(\ref{MDE1}) we can deduce that for any $t\geq 0$, the Jacobian matrix is given by
\begin{equation}
    \label{machine}
\begin{array}{rcl}
\displaystyle\frac{\partial G}{\partial\eta}(t,\eta)&=&  \displaystyle I_n-\int_{0}^{t}\Phi(t,s)Df(s,y(s,t,\eta))
\frac{\partial y}{\partial\eta}(s,t,\eta)\,ds \\\\
&=&I_n - \displaystyle \int_{0}^{t}\frac{d}{ds}\left\{\Phi(t,s)\frac{\partial y}{\partial \eta}(s,t,\eta)\right\}\,ds \\\\
&=&\displaystyle \Phi(t,0)\frac{\partial y(0,t,\eta)}{\partial \eta},
\end{array}
\end{equation}
and Theorems 7.2 and 7.3 from  \cite[Ch.1]{Coddington}  imply that
$Det \frac{\partial G(t,\eta)}{\partial \eta}>0$ for any $t\geq 0$. 

Finally, by using Hadamard's Theorem (see \cite{Plastock}), due to Theorem  \ref{teorema1} we know that $|G(t,\eta)|\to +\infty$ as $|\eta|\to +\infty$ combined with the fact that $\eta \mapsto G(t,\eta)$ is $C^{1}$ and its Jacobian matrix has a non vanishing determinant, it follows that $\eta \mapsto G(t,\eta)$ is a global diffeomorphism for any fixed $t\geq 0.$

\medskip
 
\noindent{\bf{Case $r = 2$.}} Due to \textbf{(P5)}  we can verify that
the second partial derivatives $\partial^{2}y(s,\tau,\eta)/\partial \eta_{j}\partial\eta_{i}$  satisfy
 the system of differential equations
\begin{equation}
\label{MDE15}
\left\{
\begin{array}{rcl}
\displaystyle \frac{d}{dt}\frac{\partial^{2}y}{\partial\eta_{j}\partial\eta_{i}}&=&\displaystyle
\{A(t)+Df(t,y)\}\frac{\partial^{2}y}{\partial\eta_{j}\partial\eta_{i}}
+D^{2}f(t,y)\frac{\partial y}{\partial\eta_{j}}\frac{\partial y}{\partial\eta_{i}} \\\\
\displaystyle \frac{\partial^{2}y}{\partial\eta_{j}\partial\eta_{i}}  &=&0,
\end{array}\right.
\end{equation}
 for any $i,j=1,\ldots,n,$ where $D^2 f$ is the formal second derivative of $f$ and  $y=y(t,\tau,\eta).$ By using (\ref{derivada-parcial}) and (\ref{MDE15}) we have
\begin{displaymath}
\begin{array}{rcl}
\displaystyle \frac{\partial^{2} G}{\partial\eta_{j}\partial \eta_{i}}(t,\eta)&=&\displaystyle -\int_{0}^{t}\Phi(t,s)D^{2}f(s,y(s,t,\eta))\frac{\partial y}{\partial \eta_{j}}(s,t,\eta)\frac{\partial y}{\partial \eta_{i}}(s,t,\eta)\,ds\\\\
& & \displaystyle
-\int_{0}^{t}\Phi(t,s)Df(s,y(s,t,\eta))\frac{\partial^{2} y(s,t,\eta)}{\partial\eta_{j}\partial\eta_{i}}\,ds\\\\
&=&\displaystyle  \Phi(t,0)\frac{\partial^{2} y(0,t,\eta)}{\partial\eta_{j}\partial\eta_{i}}.

\end{array}
\end{displaymath}  
Thus,  the map $\eta \mapsto G(t,\eta)$ is $C^{2}$ for any fixed $t\geq 0$. The identity $\xi=G(t,H(t,\xi))$ for any fixed $t\in \mathbb{R}^{+}$, the Jacobian matrix of the identity map on $\mathbb{R}^{n}$ can be seen as
\begin{displaymath}
DG(t,H(t,\xi))DH(t,\xi)=I_n \quad \textnormal{for any fixed $t\in \mathbb{R}^{+}$}.
\end{displaymath}

By Case $r=1$, we have that $\eta\mapsto G(t,\eta)$ is a diffeormorphism of class $C^{1}$ for any fixed $t\in \mathbb{R}^{+}$, which implies that
\begin{equation}
\label{Jaco2}
DH(t,\xi)=[DG(t,H(t,\xi))]^{-1} \quad \textnormal{for any $t\in \mathbb{R}^{+}$}
\end{equation}
is well defined. In addition, note that $(t,\xi) \mapsto DH(t,\xi)$ is continuous since the maps $A\mapsto A^{-1}$ and $(t,\xi) \mapsto DG(t,H(t,\xi))$ are continuous for any $A\in Gl_{n}(\mathbb{R})$ and $(t,\xi)\in \mathbb{R}^{+}\times \mathbb{R}^{n}$.

Now, differentiating again with respect to the second variable, we have the formal computation
\begin{displaymath}
D^{2}G(t,H(t,\xi))DH(t,\xi)DH(t,\xi)+DG(t,H(t,\xi))D^{2}H(t,\xi)=0
\end{displaymath}
and the identity (\ref{Jaco2}) implies that
\begin{equation}\label{Jaco3}
   D^{2}H(t,\xi)=-DH(t,\xi)D^{2}G(t,H(t,\xi))DH(t,\xi)DH(t,\xi).
\end{equation}

It is easy to see that $D^{2}H(t,\xi)$ is continuous with respect to $(t, \xi)$ due to is a composition of maps that are continuous with respect to $(t, \xi).$ Therefore $\eta \mapsto G(t,\eta)$ is a global diffeomorphism of class $C^2$ for any fixed $t\geq 0.$

\medskip

\noindent{\bf{Case $r \geq 3 $.}} By using \textbf{(P5)}  we can  conclude that $\eta \mapsto y(0,t,\eta)$ is $C^{r}$and the partial derivatives
\begin{displaymath}
(t,\eta)\mapsto \frac{\partial^{|m|} y(0,t,\eta)}{\partial\eta_{1}^{m_{1}}\cdots \partial \eta_{n}^{m_{n}}}, \quad \textnormal{where $|m|=m_{1}+\ldots+m_{n}\leq r$},
\end{displaymath}
are continuous for any $(t,\eta)\in \mathbb{R}^{+}\times\mathbb{R}^{n}$. Moreover, this fact combined with (\ref{machine}) shows that the partial derivatives up to order $r$--th  of $G$ with respect to $\eta$
\begin{displaymath}
(t,\eta)\mapsto \frac{\partial^{|m|} G(t,\eta)}{\partial\eta_{1}^{m_{1}}\cdots \partial \eta_{n}^{m_{n}}}=\Phi(t,0)\frac{\partial^{|m|} y(0,t,\eta)}{\partial\eta_{1}^{m_{1}}\cdots \partial \eta_{n}^{m_{n}}}, \quad \textnormal{where $|m|=m_{1}+\ldots+m_{n}\leq r$},
\end{displaymath}
are continuous in $\mathbb{R}^{+}\times \mathbb{R}$. Additionally, the higher formal derivatives of $H$ up to order $r-$th and its continuity on $\mathbb{R}^{+}\times \mathbb{R}^{n}$ can be deduced in a recursive way from (\ref{Jaco2}) and (\ref{Jaco3}).

\medskip

The property (iii) of Definition \ref{TopEqCr} is immersed in the previous analysis. Therefore, the result follows.
 \end{proof}
 
 \begin{remark}
  As we previously  emphasized, this result improves two facts: i) In \cite{Lin2}, F. Lin showed that the linearization between (\ref{lin}) and (\ref{nolin}) is of class $C^0$ when $f(t,0)$ is bounded, while we show that it is of class $C^r$ with $r \geq 1$ adding the fact that $f$ is bounded on $t$ and scapes to infinity as $x$ goes to infinity. ii) In \cite{CMR}, the authors proved that the linearization of class $C^r$ with $r \geq 1$ when the nonlinearity is considered is bounded, while we proved that it has the same class of regularity, however we regard unbounded nonlinearities. 
 \end{remark}
 
 \begin{remark}
 Although we follow the lines of the proof \cite[Th. 4.1]{CMR}, we pointed out that in that result it is proved that $|G(t,\eta)|\to +\infty$ as $|\eta|\to +\infty$ in order to use Hadamard's result,  while in Theorem \ref{teorema 3}  this previous fact is immediate due to Theorem \ref{teorema1} which is a consequence that we are using a weaker definition of topological equivalence than the one used in \cite{CMR}.
 \end{remark}


\begin{thebibliography}{99}
\bibitem{CMR}  \'A. Casta\~neda, P. Monz\'on, G.Robledo,
Smoothness of topological equivalence on the half--line for nonautonomous systems. Proc. Roy. Soc. Edinburgh Sect. A, 150 (2020), 2484--2502.

\bibitem{CR2018}\'A. Casta\~neda, G. Robledo,
Dichotomy spectrum and almost topological conjugacy on nonautonomous unbounded difference systems. 
Discrete Contin. Dyn. Syst. 38 (2018), 2287--2304.

\bibitem{JDE} \'A. Casta\~neda and G. Robledo, Differentiability of Palmer's linearization theorem and converse result for density functions. \emph{J. Differential Equations} 259 (2015), 4634--4650. 

\bibitem{Monzon1}
\'A. Casta\~neda, P. Monz\'on, G.Robledo, Nonuniform contractions and density stability results via a smooth topological equivalence. Dyn. Syst. 38 (2023), 197--196.

\bibitem{Coddington} E. Coddington and N. Levinson. 
\emph{Theory of Ordinary Differential Equations} (Mc Graw--Hill: New York, 1955).

\bibitem{Dragicevic2}
D. Dragi\v{c}evi\'c, W. Zhang, W. Zhang,
Smooth linearization of nonautonomous differential equations with a nonuniform dichotomy. Proc. Lond. Math. Soc. 121 (2020), 32-50.

\bibitem{Hartman} P. Hartman. 
\emph{Ordinary Differential Equations} 
(SIAM: Philadelphia, 2002). 

\bibitem{Hartman1} P. Hartman, A lemma in the theory of structural stability of differential equations. Proc. Amer. Math. Soc. 11 (1960), 610–620.

\bibitem{Huerta} I. Huerta, Linearization of a nonautonomous unbounded system with nonuniform contraction: a spectral approach. Discrete Contin. Dyn. Syst. 40 (2020), 5571--5590.

\bibitem{Jara} N. Jara, Smoothness of class $C^2$ of nonautonomous linearization without spectral conditions, J. Dynam. Differential Equations (2022), doi.org/10.1007/s10884-022-10207-5

\bibitem{Lin2} F. Lin,
Hartman's linearization on nonautonomous unbounded system.
Nonlinear Anal. 66 (2007), 38--50.

\bibitem{Palmer} K.J. Palmer. A generalization of Hartman’s linearization theorem. J. Math. Anal. Appl. 41
(1973), 753--758.

\bibitem{Plastock}
R. Plastock. Homeomorphisms between Banach spaces. Trans. Amer. Math. Soc.
200 (1974), 1691--7183.

\bibitem{Pugh}
C. Pugh, On a theorem of P. Hartman.
Amer. J. Math. 91 (1969), 363–367.
\end{thebibliography}
\end{document}